\documentclass[11pt]{article}
\usepackage{amsfonts,amsmath,amssymb,amsthm, amscd}
\usepackage{graphicx}

\numberwithin{equation}{section}

\newtheorem{theorem}{Theorem}[section]
\newtheorem{prop}[theorem]{Proposition}
\newtheorem{lemma}[theorem]{Lemma}
\newtheorem{cor}[theorem]{Corollary}

\theoremstyle{definition}
\newtheorem{definition}[theorem]{Definition}
\newtheorem{example}[theorem]{Example}
\newtheorem{remark}[theorem]{Remark}
\newtheorem{question}[theorem]{Question}

\def\<{{\langle}}
\def\>{{\rangle}}
\def\G{{\Gamma}}

\def\e{{\epsilon}}
\def\g{{\gamma}}
\def\S{\mathbb S}
\def\R{\mathbb R}

\def\l{{\lambda}}

\def\Rd{{\cal R}_d}

\def\e{\epsilon}

\def\ni{\noindent}
\def\bs{\bigskip}
\def\Z{\mathbb Z}

\def\D{{\Delta}}

\begin{document}

\title{Invariants of Links in Thickened Surfaces}

\author{J. Scott Carter \and Daniel S. Silver \and Susan G. Williams\thanks{The second and third
authors were partially supported by grants \#245671 and \#245615 from the Simons
Foundaton.} \\ {\em
{\small Department of Mathematics and Statistics, University of South Alabama}}}

\maketitle 

\begin{abstract}  A group invariant for links in thickened closed orientable surfaces is studied. Associated polynomial invariants are defined.
The group detects nontriviality of a virtual link and determines its virtual genus.


\end{abstract}

\section{Introduction} A \emph {link in a thickened surface} is a closed 1-dimensional submanifold $\ell =\ell_1 \cup \cdots \cup \ell_d \subset S \times I$, where $S$ is a closed, connected orientable surface.  Two such links $\ell, \ell' \subset S \times I$ are \emph{equivalent} if there exists an orientation-preserving homeomorphism 
$$h: (S \times I; S \times \{0\}, S \times \{1\}, \ell) \to (S \times I; S \times \{0\}, \ell'). $$ Equivalent links are regarded as the same. 

A link $\ell \subset S \times I$ is \emph{trivial} if its components bound pairwise disjoint embedded disks. An \emph{oriented link} is defined in the usual way by giving an orientation to each component of $\ell \subset S \times I$. The homeomorphism $h$ is required to preserve all orientations.  A \emph{knot} is a link with only one component. 
Links in $\S^2 \times I$ correspond bijectively to isotopy classes of (classical) links in $\S^3$.

Our purpose is to introduce a group and associated polynomial invariants for links $\ell$ in thickened surfaces $S\times I$. It is well known that $\ell$ represents a virtual link. We show that the group associated to $\ell$ detects nontriviality of the virtual link  (Theorem \ref{trivial}) as well as virtual genus (Theorem \ref{genusthm}).

We are grateful to Josh Barnard and Yorck Sommerh\"auser for helpful comments. 


\section{The covering group of a link in a thickened surface.}  Let $\ell = \ell_1 \cup \cdots \cup \ell_d$ be a link in a thickened closed orientable surface $S \times I$.  The universal cover $\tilde S$ of $S$ has deck transformation group $\G = \pi_1 S$. When the genus of $S$ is positive, $\tilde S$ is homeomorphic to  $\mathbb R^2$. The link $\ell$ lifts to $\tilde \ell \subset \tilde S \times I$. Equivalently, one can lift a diagram $D$ for $\ell$ to $\tilde D \subset \tilde S$. When $S$ is a torus, $\tilde D$ is a ``doubly-periodic textile structure" in the sense of \cite{mg09}

\begin{figure}
\begin{center}
\includegraphics[height=2.4 in]{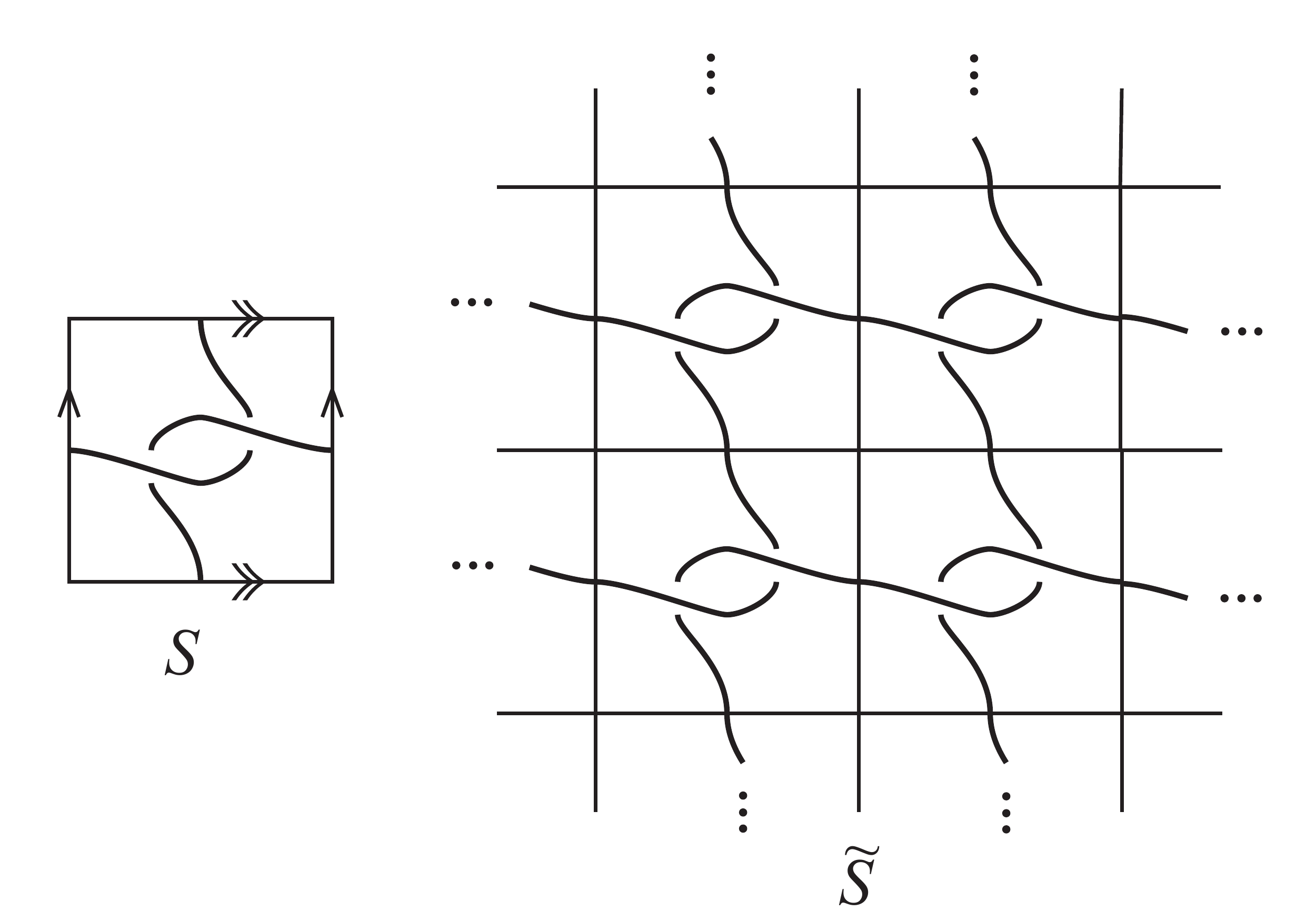}
\caption{Diagram $D$ and lift $\tilde D$}
\label{textile}
\end{center}
\end{figure}

We consider the fundamental group $\pi_1(\tilde S \times I \setminus \tilde \ell)$. A homeomorphism $h: S \times I \to S \times I$  taking one link to another lifts to the universal covers and induces an isomorphism of the corresponding groups. Hence $\pi_1(\tilde S \times I \setminus \tilde \ell)$ is an invariant of $\ell \subset S \times I$. 

\begin{definition} If $\ell \subset S\times I$ is a link in a thickened closed orientable surface, then its \emph{covering group} is $\pi_1(\tilde S \times I \setminus \tilde \ell)$. It is denoted by $\tilde\pi_\ell$. \end{definition}  

\begin{remark} \label{2sphere} When $S = \S^2$, the covering group $\tilde\pi_\ell$ is the classical link group $\pi_1(\S^3\setminus \ell)$. 

\end{remark} 

We assume throughout that $S$ is a closed orientable surface with {\sl positive} genus.
The nontrivial group $\G$ acts on $\tilde\pi_\ell$, and we write $a^\g$ for the image of $a\in\tilde\pi_\ell$ under $\g\in \G$. 

Once an orientation for $\ell$ is chosen, an orientation for $\tilde \ell$ can be lifted. We choose a basepoint in $\tilde S \times \{1\} \subset \tilde S \times I \setminus \tilde \ell$, and  we use it throughout.
Wirtinger's algorithm then yields a presentation for $\tilde\pi_\ell$, with a generator corresponding to each arc of $\tilde D$, and a relator for each crossing. The presentation is infinite. 
However, the generators comprise finitely many orbits $\{ a^\g \mid \g \in \G\}$, one for each arc of $D$. Similarly, we need only a finite number of relator orbits. 

\begin{lemma} \label{square} Let $\ell \subset S \times I$ be a link in a thickened surface. Then $\tilde\pi_\ell$ has a presentation such that the number of generator orbits is equal to the number of relator orbits. 
\end{lemma} 

\begin{proof} In the diagram $D \subset S$, the number of arcs is greater than or equal to the number of crossings; we can obtain equality by Reidemeister moves. Each arc of $D$ corresponds to a generator orbit in the presentation of $\tilde\pi_\ell$ described above, and each crossing to a relator orbit. 
\end{proof} 

The proof of Lemma \ref{square} suggests a form of presentation for $\tilde\pi_\ell$ that we will use throughout. 

Choose a fundamental domain $R$ for the surface $S$, a $2g$-gon. If a boundary edge of $R$ intersects the diagram $D$ for $\ell$, we can assume that it does so transversely and in its interior.  We can also assume that every component of $D$ contains an under-crossing.
In $R$, select representatives $a_1, \ldots, a_n$ of the $\Gamma$-orbits of arcs, which we identify with meridianal generators in $\tilde\pi_\ell$. 
The edges of $R$ can be oriented and ordered so that they project in $S$ to a set of generators  $x^*_1, y^*_1,\ldots x^*_g, y^*_g$ for $\pi_1(S)= \G$. We label edges (in pairs) with the \emph{dual}  generators $x_1, y_1,\ldots x_g, y_g$ (see Example 3.31 of \cite{hatcher} or page 83 of \cite{cz93}), and we choose these as generators of $\G$. A deck transformation corresponding to a generator, say $x_i^*$, takes $R$ to a contiguous region to the right of an oriented edge labeled $x_i$.  Each $\g\in \G$ carries arcs $a_1, \ldots, a_n$  to arcs identified with $a_1^\g, \ldots, a_n^\g$.  Some of these translated arcs may also intersect the fundamental domain $R$. We write Wirtinger relators $r_1, \ldots, r_m$ corresponding to the crossings in $R$ in the usual fashion. Then $\tilde\pi_\ell$ is presented by the collection of generators $a_i^\g$ and relators $r_j^\g$.

We denote the presentation described above by $\< a_1, \ldots, a_n \mid r_1, \ldots, r_m\>_\G$.  We may regard this either as shorthand for an infinite group presentation, or as an operator group presentation.  Operator groups are discussed in detail in the next section.

Figure \ref{group} below illustrates with a simple example. 

\begin{figure}
\begin{center}
\includegraphics[height=2.5 in]{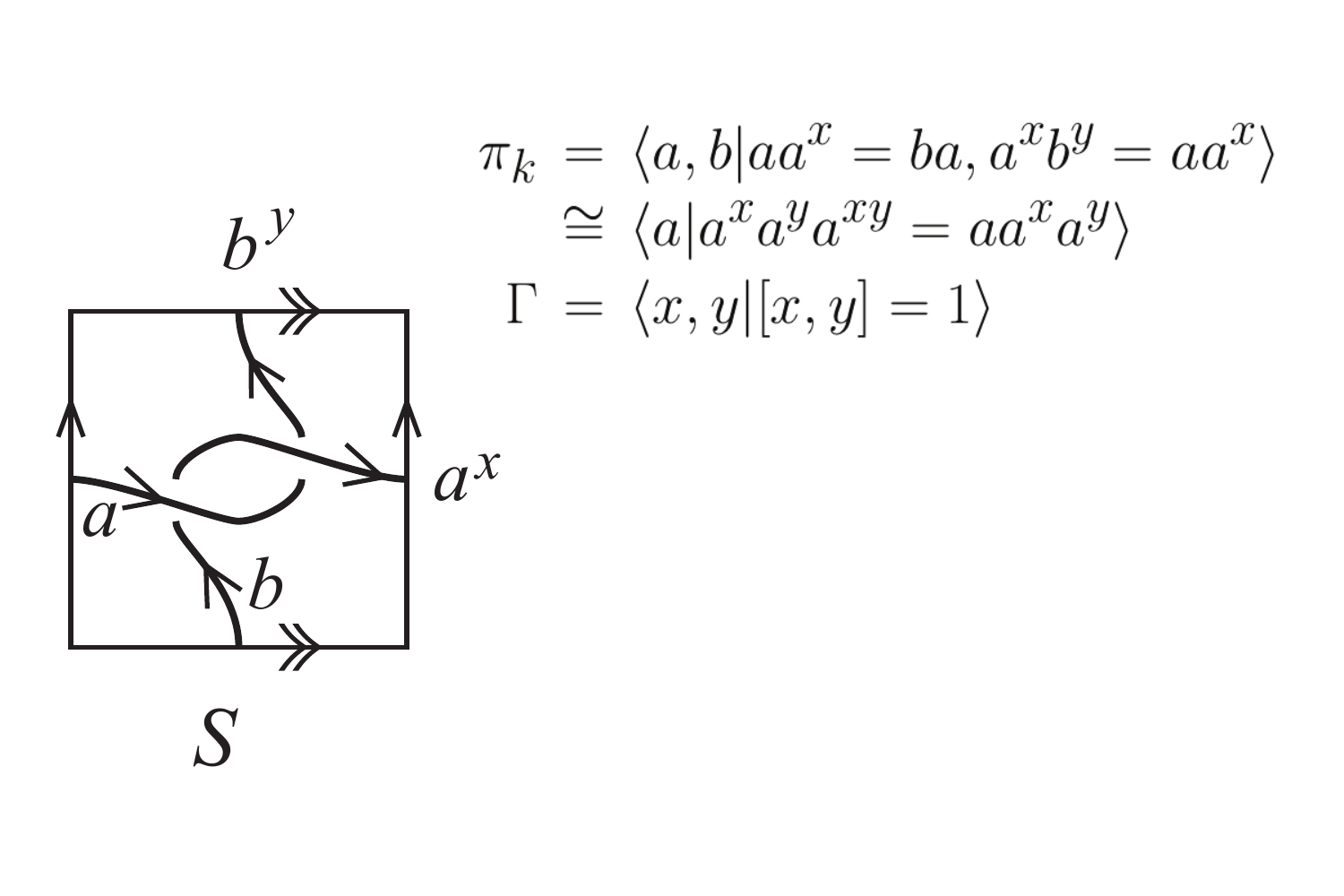}
\caption{Group $\tilde\pi_k$ of a knot $k$ in a thickened torus}
\label{group}
\end{center}
\end{figure}

The groups $\tilde\pi_\ell$ and $\pi_1(S \times I \setminus \ell)$ are, of course, related. 
Given an orbit presentation $P$ of $\tilde\pi_\ell$ as above, obtain a group presentation $\hat P$ by introducing the generators $x_1, y_1,\ldots x_g, y_g$ and relator $\Pi_{i=1}^g [x_i, y_i]$ and replacing any symbol $a^\g, \g \in \G$, appearing in the relators with $\g a \g^{-1}$. 

\begin{prop}\label{relate} $\hat P$ is a presentation of $\pi_1(S \times I \setminus \ell)$.   \end{prop}

\begin{proof} Consider the short exact sequence
$$1 \to \tilde \pi_\ell \buildrel p_* \over \longrightarrow \pi_1(S \times I \setminus \ell)\buildrel q \over \longrightarrow  \G \to 1$$
induced by the covering projection $p: \tilde S \times I \setminus \tilde \ell \to S\times I \setminus \ell$. 
The natural homeomorphism from $S$ to $S \times \{1\} \subset S\times I \setminus \ell$ induces a  splitting $s: \G \to \pi_1(S\times I \setminus \ell)$, and hence $\pi_1(S \times I \setminus \ell)$ is a
semidirect product $\tilde \pi_\ell \rtimes_\theta \G.$ (When $\g \in \G$, we abbreviate the image $s(\gamma)$ by $\gamma$ for notational simplicity.) Let $a_1^\eta, a_2^\eta, \ldots, a_n^\eta\ (\eta \in \G)$ be generators of  the covering group $\tilde \pi_\ell$. Then the group $\pi_1(S\times I \setminus \ell)$ has a presentation of the form $$\<\tilde \pi_\ell, \G \mid \g a_i^\eta \g^{-1}=  \theta_\g(a_i^\eta)   \>,$$ where $\eta, \g$ range over $\G$ and $i=1, \ldots, n$. By the definition of the covering group, $\theta_\g(a_i^\eta)$ is equal to $\g\eta a_i \eta^{-1}\g^{-1}$. When $\g$ is the identity element $e \in \G$, the relations imply that  each $a_i^\eta$ is equal to $\eta a_i \eta^{-1}$. (The symbol $a_i$ is a shorthand for $a_i^e$.) We apply Tietze transformations to eliminate the generators $ a_i^\eta, \eta \ne e$. The remaining 
relations $\g a_i^\eta \g^{-1} = \theta_\g(a_i^\eta)$ then become redundant and we remove them as well. The resulting presentation for $\pi_1(S \times I \setminus \ell)$ is  $\hat P$. 
\end{proof}

Proposition \ref{relate} immediately yields the following fact, provable also by appealing to a short exact sequence in homology. 

\begin{cor} \label{homology} $H_1(S \times I \setminus \ell; \Z) \cong H_1 (S; \Z) \oplus \Z^d$, where $\Z^d$ is generated by the classes of meridians of $\ell$, one from each component.

\end{cor}

\begin{remark} \label{spatial graph} $S \times I$ is the exterior of a spatial graph consisting of a pair of disjoint, standardly embedded $g$-leafed roses $\vee_i X_i, \vee_i Y_i \subset \S^3$ such that ${\rm lk}(X_i, Y_j) = \delta_{i,j}$.  Hence $S \times I \setminus \ell$ is the exterior of the spatial graph $\Gamma = (\vee_i X_i) \cup  (\vee_i Y_i) \cup \ell$, where $\ell$ is disjoint from the circles $X_i, Y_i$. When $g=1$, $\Gamma$ is a classical link (cf. \cite{mg09}).  In the presentation $\hat P$ of Proposition \ref{relate}, we may regard $x_i$ as the class of $X_i$. However, $y_i$ is not in general conjugate to the class of $Y_i$. 
\end{remark}

\section{Operator groups.} The covering group $\tilde\pi_\ell$ is an example of an operator group, a notion introduced by Krull and Noether. 
Additional material can be found in \cite{bourbaki}, \cite{kurosh}, \cite{robinson}.

\begin{definition} An \emph{operator group} is a pair $(\pi, \G)$ and 
a function $\pi \times \G \to \pi,\ (g, \g) \mapsto g^\g$, such that \begin{enumerate}
\item $\pi$ is a group; 
\item $\G$ is a set (the ``operator set"); 
\item $\forall \g \in \G$, the map $g \mapsto g^\g$ is an endomorphism of $\pi$. \end{enumerate}
\end{definition}

\begin{remark} When $\G$ is empty, $(\pi, \G)$ is a group in the usual sense. 
\end{remark}

In the operator groups that we consider, $\G$ is itself a group. We assume additional structure: \bs

4. $g^{\g \eta} = (g^\g)^\eta\ \forall  g\in \pi, \g, \eta \in \G$ \bs

\ni The abelianization of $\pi$ can then be regarded in a natural way as a right $\Z[\G]$-module.

\begin{definition} Let $(\pi, \G)$ and $(\bar \pi, \bar \G)$ be operator groups and $\G, \bar \G$ groups. A {\it homomorphism} $(f, \phi): (\pi, \G) \to (\bar \pi, \bar \G)$ consists of group homomorphisms $f: \pi \to \bar \pi$ and $\phi: \G \to \bar \G$ such that $$f(g^\g) = f(g)^{\phi(\g)}\ \forall g \in \pi, \g \in \G.$$  An {\it isomorphism} is a homomorphism $(f, \phi)$ such that both $f$ and $\phi$ are isomorphisms.  \end{definition}

Henceforth we regard $\tilde\pi_\ell$ as an operator group with $\G = \pi_1(S)$. 
If $\ell, \ell' \subset S \times I$ are equivalent links, then there exists an isomorphism from $\tilde\pi_\ell$ to $\tilde\pi_{\ell'}$. The automorphism $\phi: \G \to \G$ can in fact be chosen to be an automorphism that is induced by a self-homeomorphism of $S$, as we see next.

Let ${\rm Aut}_h(\G)$  denote the subgroup of ${\rm Aut}(\G)$ consisting of automorphisms induced by orientation-preserving homeomorphisms of $S$. 

\begin{theorem} Let $\ell, \ell'\subset S \times I$ be equivalent links in a thickened surface. There exists an isomorphism $(f, \phi): (\tilde\pi_\ell, \G) \to (\tilde\pi_{\ell'}, \G)$ such that $\phi \in {\rm Aut}_h(\G)$. \end{theorem}

\begin{proof} Assume that there exists an orientation-preserving homeomorphism $h: (S \times I, S\times \{i\}) \to (S \times I, S\times \{i\})$, $i= 0, 1$, taking $\ell$ to $\ell'$. Then $h$ restricts to an orientation-preserving homeomorphism of $S \times \{1\}$, which we identify with $S$. Without loss of generality, we can assume that $h$ leaves fixed the basepoint $* \in S$.

The map $h$ lifts to a homeomorphism $\tilde h$ of $\tilde S \times I$ that leaves fixed a lift $\tilde *$ of the point $*$. It induces an isomorphism 
$\tilde f: \pi_1( \tilde S \times I \setminus \tilde \ell, \tilde *) \to \pi_1( \tilde S \times I \setminus \tilde \ell', \tilde *)$ and also an automorphism 
$\phi$ of $\pi_1(S, *)$. The pair $(f, \phi)$ determines 
an isomorphism from $\tilde\pi_\ell$ to $\tilde\pi_{\ell'}$. 
\end{proof} 

\begin{theorem}\label{trivial} A link $\ell= \ell_1 \cup \cdots \cup \ell_n$ in a thickened surface $S \times I$ is trivial if and only if $\tilde\pi_k \cong \<a_1, \ldots, a_n \mid \>_\G$. \end{theorem}

\begin{proof} If $\ell$ is trivial, then clearly $\tilde\pi_\ell \cong \<a_1, \ldots, a_n \mid \>_\G$. 

Conversely, assume that $\tilde\pi_\ell \cong \<a_1, \ldots, a_n \mid \>_\G$. It suffices to prove that any component $\ell_i$ bounds a disk that does not intersect any of the other components. 

The group $\G$ acts freely on 
$H_1(\tilde S \times I \setminus \tilde \ell; \Z)$, which is freely generated by the classes of meridians of distinct components of $\tilde \ell$, with orientations induced by a fixed orientation of $\ell$. 
  
Fix $i \in \{1, \ldots, n\}$, and choose a {\it longitude} for $\ell_i$, an oriented simple closed curve $\l$ in the boundary of a tubular neighborhood $N_i$ of $\ell_i$ and intersecting a meridian $m$  transversely in a single point. (The homotopy class of the longitude $\l$ is not unique.) Then $\l$ together with a base path represents an element of $\G = \pi_1(S \times I)$. For notational convenience, $\l$ will also denote this element. 

Let $\tilde m$ be any meridian of the preimage $\tilde \ell_i$ of $\ell_i$. Consider the class $[\tilde m] \in H_1(\tilde S \times I \setminus \tilde \ell; \Z)$. The action of $\l$ takes $\tilde m$ to another, homologous meridian of the same component of $\tilde \ell_i$, and hence it fixes the class $[\tilde m]$. Since $\G$ acts freely, either $\l = 1$ or else $[\tilde m] =0$. But $\tilde m$ is an arbitrary meridian of $\tilde \ell_i$, and meridians of distinct components of $\ell_i$ are among the set of free generators of $H_1(\tilde S \times I \setminus \tilde \ell; \Z)$. Hence $\l = 1$. We conclude that each component of $\tilde \ell_i$ is a closed curve. 

Consider any component of $\tilde \ell_i$. Lift $\l$ to $\tilde \l$ in the boundary $\partial \tilde N_i$ of a tubular neighborhood of the component. Let $\tilde m\subset \partial \tilde N_i$ be a meridian such that $\tilde \l$ and $\tilde m$ intersect transversely and in a single point. Again for notational convenience, we let $\tilde \l$ and $\tilde m$ together with base paths denote the 
elements of $\tilde\pi_\ell$ that they represent. Since $\tilde \l$ and $\tilde m$ commute and 
$\tilde\pi_\ell$ is free, $\tilde \l$ and $\tilde m$ must be powers of a common element. However, $\tilde m$ is not a proper power since its $\G$-orbit is among a set of free generators of $\tilde\pi_\ell$. Hence $\tilde \l$ is a power of $\tilde m$. Reselecting $\l$, if necessary, we can assume that $\tilde \l$ is trivial in $\tilde\pi_\ell$. Dehn's lemma implies that $\tilde \l$ bounds a properly embedded disk in the exterior of $\tilde \ell$. 

Projecting down, we find that $\l$ is null-homotopic in $S \times I \setminus \ell$. Dehn's lemma now implies that $\l$ bounds an embedded disk in the exterior of $\ell$. Since the component $\ell_i$ that we considered was arbitrary, the link $\ell$ is trivial. 
\end{proof}


\section{Polynomial invariants from the covering group.} \label{poly}  Let $\ell= \ell_1 \cup \ldots \cup \ell_d$ be an oriented link in a thickened surface $S \times I$. Let $\e: \tilde\pi_\ell \to \Z^d = \< t_1, \ldots, t_d \mid [t_i, t_j]=1 \ \forall \ i, j\>$ be the homomorphism that maps every meridian of the lift of $\ell_i$ to $t_i, 1 \le i \le d$. 
Let $K$ be the kernel of $\e$. Its abelianization $M= K/[K, K]$ is a right-module over $\Z[\G \times \Z^d]$. In order to obtain a Noetherian module, we pass to the quotient $\bar M = M/M_0$, where $M_0$ is the submodule of $M$ generated by all elements of the form $a^\gamma - a^\eta$, where $a \in M$, $\gamma, \eta \in \G$, and $\gamma \eta^{-1} \in [\G, \G]$. Then $\bar M$ is a right-module over the Noetherian ring 
$\Z[H_1\G \times \Z^d] \cong (\Z[\Z^{2g}])[t_1^{\pm1}, \ldots, t_d^{\pm 1}]$.

Denote $\Z[\Z^{2g}]$ by ${\cal R}$. 
By Lemma \ref{square}, $\bar M$ is presented by a square $n \times n$ matrix $A$ over $\Rd= {\cal R}[t_1^{\pm1}, \ldots, t_d^{\pm 1}]$. For any nonnegative integer $i$, define $\D_i(\ell)$ to be the greatest common divisor of the 
$(n-i)\times(n-i)$ minors of $A$.   We call $\D_i(\ell)$ the 
$i$th \emph{Alexander polynomial} of $\ell\subset S\times I$.

\begin{remark} (1) For convenience, we refer to elements of both $\G$ and 
$H_1 \G = \G/[\G, \G]$ as \emph{operators}. 

(2) We have assumed throughout that the genus of $S$ is positive. If we were to consider the case $S = \S^2$, then $\D_i(\ell)$ would be the usual Alexander polynomial invariants of $\ell$. \end{remark} 

The polynomials $\D_i(\ell)$ are well defined up to multiplication by units in $\Rd$ and symplectic change of coordinates in $H_1 \G$. We make this precise: 

Recall that a module $H$ over $\Z$ (resp. $\R$) is \emph{symplectic} if it is equipped with a skew-symmetric pairing $H \times H \to \Z$ (resp. $H \times H \to \R$),  $(v, w) \mapsto v\cdot w$.  A \emph{standard basis} is a basis $a_1, b_1, \ldots, a_g, b_g$ of $H$ such that $a_i\cdot a_j = b_i \cdot b_j =0$ and $a_i\cdot b_j =\delta_{i,j}$ for all $i, j$. The first homology group of any compact oriented surface is a symplectic module, and a standard basis exists that is represented by simple closed oriented circles.

Fix a standard basis for $H_1(S; \Z) \cong \Z^{2g}$. Any $\phi \in {\rm Aut}_h(\G)$ induces an element of the symplectic group  ${\rm Sp}(2g, \Z)$, and hence an automorphism $\phi_\sharp$ of $\Rd$ by extending linearly in ${\cal R}$ and mapping each $t_i$ to itself. Two polynomials $\D, \D'$ are 
\emph{equivalent} if $\D' = u\cdot \phi_\sharp(\D)$ for some unit $u\in \Rd$ and some $\phi \in {\rm Aut}_h(\G)$. \bs

\begin{prop} \label{orient} Assume that $\ell'$ is obtained from $\ell$ by reversing the orientation of the $j$th component. Then, for any $i\ge 0$,  $\D_i(\ell')$ is obtained from $\D_i(\ell)$ by replacing $t_j$ with $t_j^{-1}$. \end{prop} 

\begin{proof} Changing the orientation of some component of $\ell$ alters the covering group by inverting generators corresponding to meridians of the component. The conclusion follows using standard Fox calculus as for classical links in the 3-sphere.    \end{proof}

\begin{example}
Returning to example of Figure 4, $$A = \begin{pmatrix}1+t x-t & -1\\ x-1-t x & t y \end{pmatrix}$$ and 
$$\D_0(k) = (xy - y) t^2 + (y- x)t + (x-1).$$ Here we write $t$ instead of $t_1$ and $x, y$ instead of $x_1, y_1$.  (In later examples, we avoid subscripts in a similar fashion.) A Dehn twist induces $\phi_\sharp: x \mapsto x y, 
y \mapsto y$. Hence $\D_0(k)$ is equivalent to $(x y^2 - y)t^2 + (y- x y) t + (x y -1)$. \end{example}

Consider the projection $q: \Rd \to \Z[t_1^{\pm 1}, \ldots, t_d^{\pm 1}]$ induced by the trivial homomorphism $\G \to \{1\}$. 

\begin{prop}\label{augment}  If $k$ is any oriented knot in a thickened surface, then $$q(\D_0(k)) = 0.$$\end{prop}

\begin{proof} If we choose the rows of $A$ to correspond to Wirtinger relations, then setting each element of $\G$ equal to $1$ will make the entries of each row sum to zero. Hence the determinant vanishes. \end{proof}

\begin{question} Does the conclusion of Proposition \ref{augment} hold for links of more than one component? 
\end{question}

\section{Symplectic rank.} Let $V$ be a submodule of $H_1\G \cong \Z^{2g}$. Tensoring with $\R$, we obtain a subspace $W = V \otimes \R$ of $H_1\G \otimes \R \cong \R^{2g}$.  

\begin{definition} The \emph{symplectic rank} of $V$, denoted by $rk_s(V)$,  is the dimension of
$W/ W\cap W^\perp$, where $W^\perp = \{v \in \R^{2g}\mid v\cdot w=0 \ \forall w \in W\}$. \end{definition}

\begin{remark} It is not difficult to see that $rk_s(V)$ is the dimension of a maximal symplectic subspace of $\R^{2g}$ contained in $W$. \end{remark}

Let $\tilde \pi_\ell$ denote the covering group of a link $\ell \subset S \times I$ in a thickened surface. As above, we regard $\tilde \pi_\ell$ as a $\G$-operator group.

\begin{definition} Let $P$ be a presentation of $\tilde \pi_\ell$. Its \emph {symplectic rank} $rk_s (P )$ is the symplectic rank of the submodule $W_P$ of $H_1\G$ generated by the operators that appear in relators. 

The \emph {symplectic rank} of $\tilde \pi_\ell$ is the minimum of $rk_s(P)$, taken over all presentations $P$ of $\tilde \pi_\ell$. It is denoted by $rk_s(\tilde \pi_\ell)$.
\end{definition}

\begin{definition} The \emph {symplectic rank} of $\D_0(\ell)$ is the symplectic rank of the submodule $W_\D$ of $H_1\G$ generated by quotients of operators that appear in the coefficients of $\D_0(\ell)$. It
is denoted by $rk_s(\D_0(\ell))$. \end{definition}

\begin{prop}\label{srank} The symplectic rank of $\D_0(\ell)$ is well defined and independent of the orientation of $\ell$. Moreover,
$$rk_s(\D_0(\ell)) \le rk_s(\tilde \pi_\ell).$$ \end{prop}

\begin{proof} Recall that $\D_0(\ell)$ is defined up to multiplication by units in $\Rd = \Z[H_1\G \times \Z^d]$ and symplectic automorphisms of $H_1\G$. 
Since $W_\D$ is spanned by {\sl quotients} of elements of $H_1\G$, it is unchanged if $\D_0(\ell)$ is multiplied by a unit of $\Rd$. Furthermore, a symplectic automorphism of $H_1\G$ preserves orthogonality and hence it takes $W_\D/ W_\D\cap W_\D^\perp$ to an isomorphic module.  The symplectic rank of $\D_0(\ell)$ is therefore well defined. By Proposition \ref{orient}, it is independent of the orientation of $\ell$. 

To see why the inequality holds, consider any $\G$-operator group presentation $P$ of $\tilde \pi_\ell$. We construct a square matrix $\tilde M$ as above with determinant equal to $\D_0(\ell)$. Any operator that appears in the polynomial must be contained in $W_P$. \end{proof}

\begin{remark} (1) We can ``base" $\D_0(\ell)$, multiplying by 
a unit of $\Rd$ so that some coefficient is monic. Then considering quotients of elements is no longer necessary. We will do this in the examples that follow. 

(2) We will see in Example \ref{sharp} that the inequality of Proposition \ref{srank} can be strict. \end{remark}


\section{Applications to virtual links.} The notion of a virtual link is due to L. Kauffman \cite{kauf99}. It is a nontrivial extension of the classical theory of knots and links. Virtual links correspond bijectively to abstract link diagrams, introduced by N. Kamada in \cite{kam93}, \cite{kam97} (see  \cite{kamkam00}).

It is shown in \cite{cks} that one can regard a virtual link as a link diagram in a closed orientable surface up to Reidemeister moves on the diagram, orientation-preserving homeomorphisms of the surface and adding or deleting hollow $1$-handles in the complement of the diagram. Adding a $1$-handle (``stabilization") is a surgery operation, removing two open disks disjoint from the diagram, and then joining the resulting boundary components by an annulus. Deleting a $1$-handle (``destabilization") is also a surgery operation, removing the interior of a neighborhood of a simple closed curve that misses the diagram, and then attaching a pair of disks to the resulting boundary. 

In general we do not assume the surface is connected, but we do assume that each component of the surface meets the link.  We say a virtual link is \emph{split} if it has a diagram $D$ supported by a 2-component surface $S$ such that each component of $S$ meets $D$.  We will also call a link $\ell \subset S \times I$ \emph{split} if it represents a split virtual link.  

The \emph{virtual genus} of $\ell$ is the minimal genus of a surface that contains a diagram representing the link.  For a non-connected surface, this is defined to be the sum of the genera of the components.

We can regard a virtual link also as an equivalence class of embedded links in thickened surfaces. 
The equivalence relation is generated by isotopy as well as stabilization/destabilization. As in \cite{kup03}, destabilization consists of parametrized surgery along an embedded annulus $A$ that is \emph{vertical} in the sense that $A = p_1^{-1}(p_1(A))$, where $p_1$ is the first-coordinate projection on $S\times I$ (see \cite{wald68}). The reverse operation of stabilization, which need not concern us here,  is a parametrized connected-sum operation with a thickened torus.

The main theorem of \cite{kup03} states that every virtual link has a unique representative $\ell\subset S\times I$ for which the genus of $S$ is equal to the virtual genus of $\ell$ and the number of components of $S$ is maximal. Uniqueness is up to Reidemeister moves and orientation-preserving homeomorphisms of the surface.  Consequently, the Alexander polynomials $\D_i(\ell)$ of a link in a thickened surface of minimal genus and maximal number of components are invariants of the virtual link it represents.

The main result of this section is the following theorem.

\begin{theorem} \label{genusthm} Let  $\ell$ be a non-split virtual link. For any representative  $\ell\subset S\times I$, the symplectic rank of $\tilde \pi_\ell$ is twice the virtual genus of $\ell$. \end{theorem}

Proposition \ref{srank} immediately yields the following.

\begin{cor} \label{reduce} For $\ell$ as above, the virtual genus of $\ell$ is at least half the symplectic rank of  $\D_0(\ell)$.
\end{cor}

The \emph{exterior} $X$ of $\ell$ is $S \times I$ minus the interior of a regular neighborhood of $\ell$. 

\begin{prop} Assume that $\ell$ is neither a split link nor a a local link (that is, a link in a 3-ball). Then the exterior $X$  is an irreducible 3-manifold with incompressible boundary.
\end{prop}

\begin{proof}

Since $\tilde S \times I$ is irreducible, so is $S \times I$ (see, for example, Proposition 1.6 of \cite{hatcher2}). An embedded 2-sphere $\Sigma \subset X$ must bound a ball in $S \times I$.  The hypotheses ensure that such a ball is in $X$.

The boundary of $X$ is incompressible if the inclusion map of any component induces an injection of fundamental groups. This is clear for each component $S \times \{j\}$, $j = 0, 1$, since 
the each inclusion map $S \times \{j\} \hookrightarrow S \times I$ induces an isomorphism of fundamental groups. 
Consider a neighborhood $N_i$ of some component $\ell_i$ of $\ell$. If $\partial N_i \hookrightarrow X$ induces a homomorphism of fundamental groups that is not injective, then by the Loop Theorem, there exists an embedded $2$-disk $D \subset X$ such that the boundary of $D$ is an essential simple closed curve in $\partial N_i$. Elements of the first homology of $\partial N_i$ can be written 
$\alpha [\l] + \beta [m]$, where $\l$ and $\ell_i$ cobound an annulus in $N_i$, $m$ is a meridian of $\ell_i$, and $\alpha, \beta$ are relatively prime integers. Corollary \ref{homology} implies that $(\alpha, \beta) = (\pm 1, 0)$. Then by thickening $D$ and adjoining it to $N_i$, we obtain a $3$-ball in $X$ containing $\ell_i$ but no other component of $\ell$. Hence $\ell$ is either a split link or a local knot, contrary to our hypothesis. Hence the boundary of $X$ is incompressible. 
\end{proof}

A curve in $S$ is \emph{homologically essential} if it represents a nontrivial element of $H_1(S; \Z)$.  We will say that a diagram $D \subset S$ of a link $\ell \subset S \times I$ is \emph{reducible} if $S$ contains a homologically essential simple closed 
curve $C$ that is disjoint from $D$. In this case, we can perform $1$-surgery on $C$ and obtain a diagram in a surface of smaller genus.

\bigskip We now prove Theorem \ref{genusthm}.

\begin{proof}  It is clear that any component of $\ell$ contained in a $3$-ball can be removed without affecting the virtual genus of $\ell$ or the symplectic rank of $\tilde \pi_\ell$. Hence we assume without loss of generality that $\ell$ is neither a split nor a local link. 

The proof of the main theorem of \cite{kup03} shows that if $\ell$ is represented by a diagram in a surface $S$ and if ${\rm genus}(S) = {\rm virtual\  genus}(\ell) + n$, for some positive integer $n$, then, after Reidemeister moves, there exists an essential $n$-component $1$-manifold $C$ that is disjoint from the diagram and along which we can perform surgery to produce a surface of genus equal to the virtual genus of $\ell$.

Build a fundamental region for $S$ by cutting along the $1$-manifold $C$ and continuing. The edges of $C$ correspond to generators of 
$\G$ that do not appear in the corresponding operator group presentation $P$ of $\tilde\pi_\ell$ and so do not appear in $W_P$.  These $n$ generators represent mutually orthogonal  elements of $H_1(S; \Z)$ since surgery along $C$ reduces the genus of $S$ by $n$.  Hence the symplectic rank of $\tilde \pi_\ell$ is at most twice the virtual genus of $\ell$.

Now suppose that $\tilde\pi_\ell$ has symplectic rank less than twice the virtual genus of $\ell$.  Then some operator group presentation $P$ of $\tilde \pi_\ell$ must omit a generator of $\G = \<x_1, y_1, \ldots, x_g, y_g \mid \prod_{i=1}^g [x_i, y_i]\>$. Without loss of generality, we can assume that the omitted generator is $x_1$. By  Proposition \ref{relate}, the group $\pi_1(S \times I \setminus \ell)$ has a presentation in which $x_1$ occurs only in the relator $\prod_i [x_i, y_i]$. Express the relator as 
$$ x_1 y_1 x_1^{-1} =( \prod_{i=2}^g[x_i, y_i])y_1.$$
Let $B$ be the subgroup of $\pi_1(S \times I \setminus \ell)$ generated by $y_1, x_2, y_2, \ldots, x_g, y_g$. Let $U$ and $V$ be the cyclic subgroups of $B$ generated by $y_1$ and $(\prod_{i=2}^g[x_i, y_i])y_1$, respectively. Since the inclusion $S \times \{0\} \to S \times I \setminus \ell$ induces an injection of fundamental groups, the subgroups $U$ and $V$ are in fact infinite cyclic. Hence $\pi_1(S \times I \setminus \ell)$ has an HNN decomposition with stable letter $x_1$, base group $B$ and infinite cyclic amalgamating subgroups $U$ and $V$ (see \cite{ls}, for example).

Since $\pi_1(S \times I \setminus \ell)$ splits over the infinite cyclic group $U$ and $X$ is irreducible with incompressible boundary, the proof of Satz 1.2 of \cite{wald67} (see also Corollary 1.2 of \cite{scott}) shows that 
there exists a proper annulus $A \subset X$ such that :\\

\item{} (1) The inclusion map $i: A \hookrightarrow X$ induces an injection $i_*: \pi_1A \to \pi_1 X$ with the image of $i_*$ conjugate to a subgroup of $U$. \\

Since the image of $i_*$ is generated by a simple closed curve in the surface $S \times \{1\} \subset X$, the image is conjugate to the entire subgroup $U$. However, we will not need this.
We do, however, need the following, which follows easily from the proof in \cite{scott}: \\

\item{}(2) The annulus $A$ meets a simple closed curve representing $x_1$ transversely in a single point.\\

We argue that, after isotopy, we can find a vertical annulus $C\times I$ in $X$  such that $C \times \{1\} \subset S$ is homologically essential.  We can then perform parametrized surgery on $A$, as in \cite{kup03}, in order to reduce the genus of $S$.

Condition (2) implies that at least one boundary component of $A$ must be contained in $\partial(S \times I)$. Moreover, since $A$ is non-separating, it is impossible for both boundary circles of $A$ to lie on the same component of $\partial(S\times I)$. 

Assume that some component of $\partial A$ lies in $\partial(S\times I)$ while the other is contained in the boundary of a component $\partial N_i$ of the neighborhood $N = N_1 \cup \ldots \cup N_d$ of $\ell$. Without loss of generality, we assume that a component lies in $S \times \{1\}$. (If it is contained in $S\times \{0\}$, then the argument is similar.) By Corollary \ref{homology}, $A$ meets the boundary of $N_i$ in a longitude. (``Longitude" was defined in the proof of Theorem \ref{trivial}.)  We can use $A$ to perform an isotopy that lifts $\ell_i$ up into a collar neighborhood $Y$ of $S\times \{1\}$ containing no other component of $\ell$, and extend the annulus to the lower boundary of $Y$.  Consider the sublink $\ell'$ of $\ell$ obtained by deleting $\ell_i$. Regard $\ell'$ as a link in the closure of $S \times I \setminus Y$. Its fundamental group results from $\pi_1(S \times I \setminus \ell)$ by annihilating a meridian of $\ell_i$. Since the quotient group also splits over the infinite cyclic group $U$ generated by $x_1$, we can apply the preceding argument. After a finite number of steps, we obtain a proper annulus satisfying (1) and (2) with boundary components on $S \times \{1\}$ and $S \times \{0\}$.  

By Lemma 3.4 of \cite{wald68}, there is an isotopy of $S \times I$ that is constant on the boundary and takes $A$ to a vertical annulus $A'$.
The link $\ell$ is carried to an equivalent link, which we continue to denote by $\ell$, that is disjoint from $A'$. 

Recall that we began with a presentation $P$ of $\tilde \pi_\ell$ that omits the generator $x_1$. Parametrized surgery on the annulus $A$ produces a link $\bar \ell \subset \bar S \times I$, where the genus of $\bar S$ is one less than that of $S$.  By Lemma \ref{cutting},  we obtain a presentation $\bar P$ for $\pi_1( \bar S \times I \setminus \bar \ell)$ from $P$ by introducing relations $x_1=y_1=1$. 

It is clear that $\bar P$ has the same symplectic rank as $P$.  Hence we may repeat the above construction until the genus of the thickened surface is half the symplectic rank of $\tilde \pi_\ell$.
\end{proof} 

\begin{lemma} \label{cutting} Let $\ell \subset S \times I$ be a link in thickened surface, and assume that $A$ is a vertical annulus in $S\times I \setminus \ell$ such that $A \cap (S \times \{1\}) = C$ represents a generator $y_1$ of $\pi_1 S \cong \< x_1, y_1, \ldots, x_g, y_g \mid \prod[x_i, y_i]\>$. If $\bar \ell \subset \bar S \times I$ is the link resulting from parametrized surgery on $A$, then $\pi_1 (\bar S \times I \setminus \bar \ell)$ is isomorphic to $\pi_1 (S\times I \setminus \ell)$ modulo the normal subgroup generated by $x_1, y_1$. \end{lemma} 

\begin{proof} Let $R$ be a fundamental domain for $S$, a $2g$-gon with oriented edges labeled $x_1, y_1, \ldots, x_g, y_g$ as above. Let $S_0$ be the bounded surface that results from $S$ by cutting along the curve $C$. The universal cover $\tilde S_0$ of $S_0$ is a subsurface of $\tilde S$, a union of copies of $R$ matched along edges except those labeled $x_1$. The link $\ell$ lifts to $\ell' \subset \tilde S_0 \times I$ and $\pi_1(\tilde S \times I \setminus \ell')$ is a $\G_0$-operator group, where $\G_0$ is the subgroup of $\G$ generated by $y_1, x_2, y_2, \ldots, x_g, y_g$.  A presentation is also a presentation of $\tilde \pi_\ell$, one in which the operator $x_1$ does not appear.  The argument of Proposition \ref{relate} shows that $\pi_1( S_0 \times I \setminus \ell')$ is isomorphic to $\pi_1(S \times I \setminus \ell)$ modulo the normal subgroup generated by $x_1$. Completing the parametrized surgery introduces the relator $y_1$.

\end{proof}


\section{Examples.} 

\begin{example} 

The diagram in Figure \ref{textile} represents a virtual knot $k$ sometimes called the \emph{virtual trefoil}.  The polynomial $\D_0(k)$, computed in Section \ref{poly}, has symplectic rank 2.  Since the link has a diagram on the torus, Theorem \ref{genusthm} implies the well-known fact that the virtual genus of the knot is $1$.
\end{example}


\begin{example} \label{sharp} Let $k$ be a virtual knot. A \emph{satellite} $\tilde k$ is defined in \cite{sw12} as in the classical case by replacing $k$ by a knot $\tilde k$ in a regular neighborhood of $k$ (but not contained in a 3-ball). It is shown that if $\tilde k$ is a satellite of $k$, then the virtual genus of $\tilde k$ is equal to that of $k$. 

Consider the double $\tilde k$ of the virtual trefoil $k$ of the previous example. 
It is a special case of a satellite knot. A diagram for $\tilde k$ appears in Figure \ref{double}.  Calculation reveals that 
$$\D_0(\tilde k) = (t-1) (x y - 1)^2  .$$
The symplectic rank of $\D_0(\tilde k)$ is zero.  However, the virtual genus of $\tilde k$ is equal to that of $k$, which is 1.  Hence the inequality of Corollary \ref{reduce} is not an equality in general.

\begin{figure}
\begin{center}
\includegraphics[height=2.2 in]{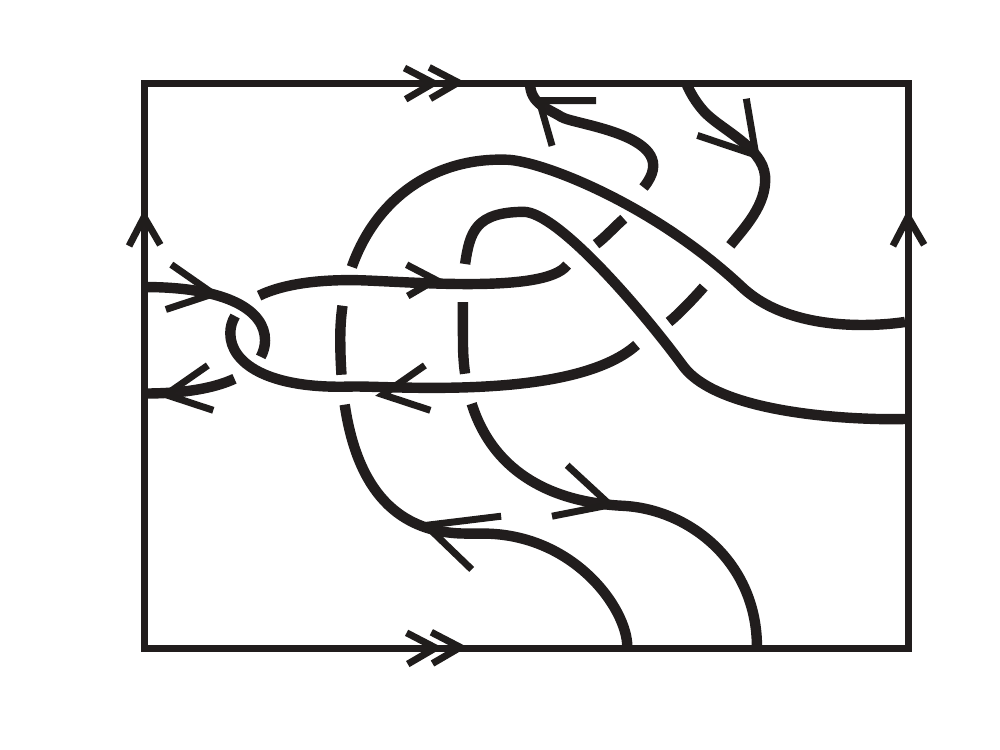}
\caption{Satellite virtual knot}
\label{double}
\end{center}
\end{figure}

\end{example}

\begin{figure}
\begin{center}
\includegraphics[height=2.8 in]{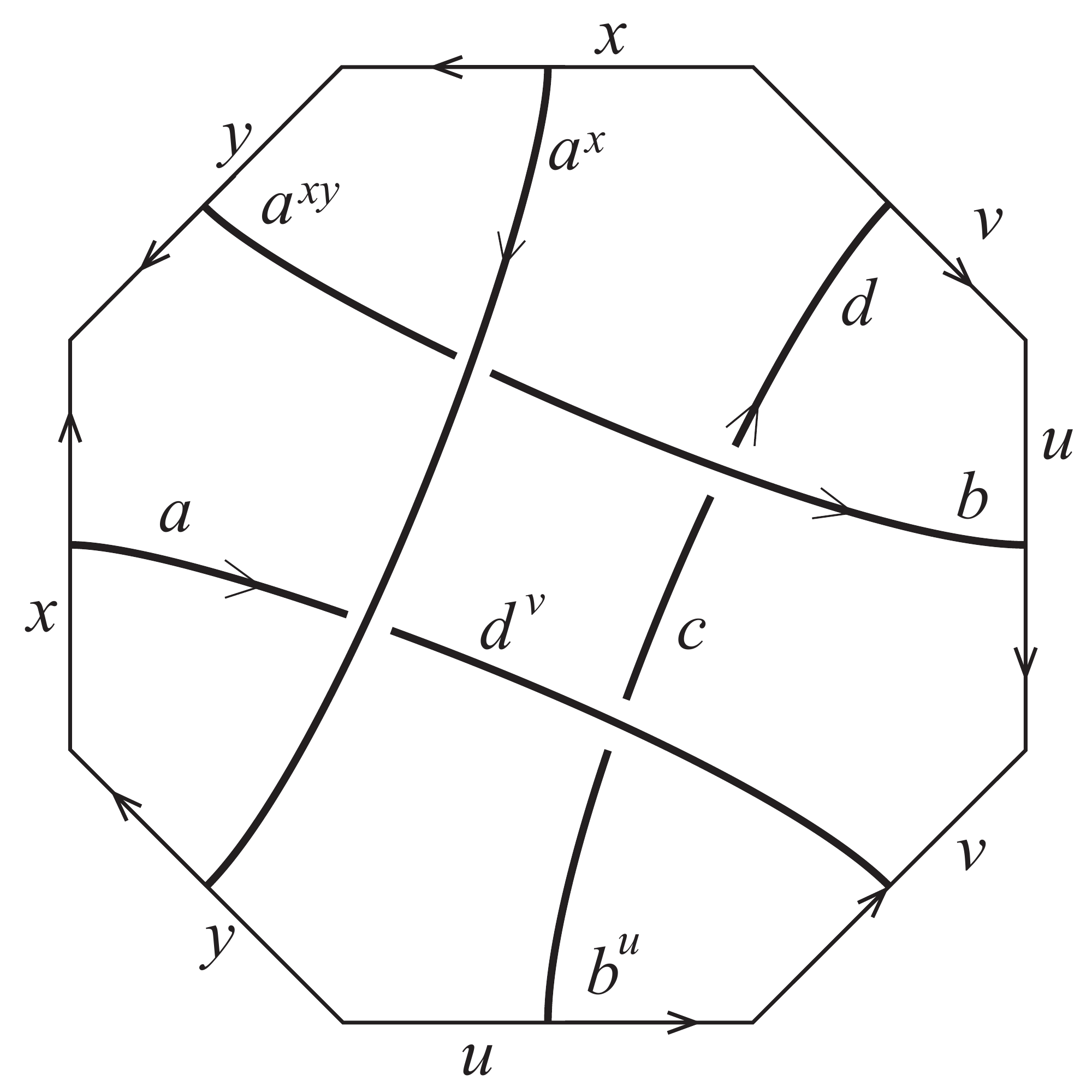}
\caption{Kishino's knot}
\label{kishino}
\end{center}
\end{figure}

\begin{figure}
\begin{center}
\includegraphics[height=3 in]{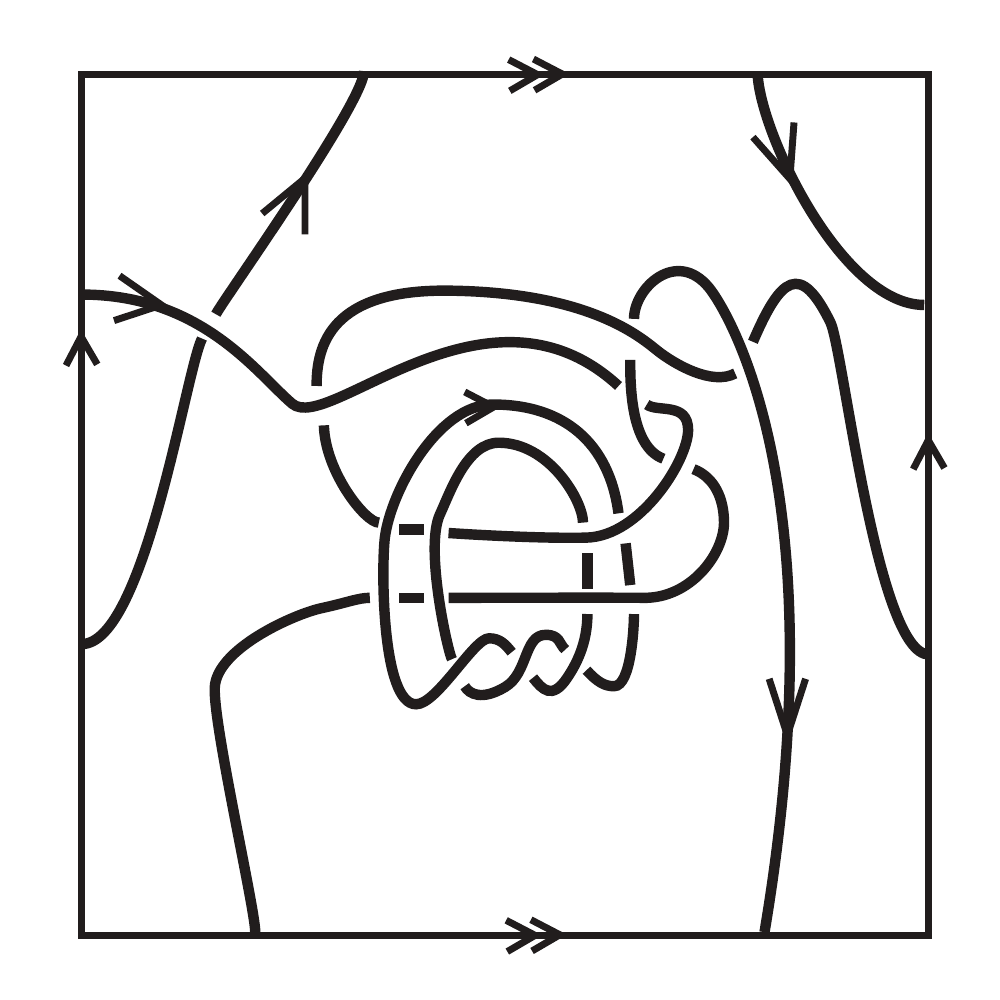}
\caption{Stoimenow's link}
\label{stoimenow}
\end{center}
\end{figure}

\begin{example} Consider the oriented diagram for Kishino's knot in Figure \ref{kishino}. The group is 
$$\tilde\pi_k = \< a, b, c, d \mid a^x b = a^{xy} a^x, a^x d^v = a a ^x, b d = c b, d^v b^u= c d^v\>_\G.$$ The associated matrix $A$ is 
$$\begin{pmatrix} x - x y - x t & t & 0 & 0 \\ x - x t - 1 & 0 & 0& v t\\ 0 & 1-t& -1 & t \\ 0 & u t& -1 & v - v t\end{pmatrix}.$$
Here $\D_0(k) = (x- u v x)t^2 + (1 + v - x + uvx- v xy- uvxy)t + (-v + vxy)$.
The symplectic rank of $\D_0(k)$ is $4$. By Corollary \ref{reduce}, the virtual genus of Kishino's knot is 2. This result was proved earlier by Kauffman and Dye \cite{dk05}, using the Jones polynomial and symplectic algebra to produce lower bounds on virtual genus.  

A virtual link $\ell$ is \emph{invertible} if some oriented diagram is equivalent to the same underlying diagram with the opposite orientation. In this case, $\D_0(\ell)(t_1, \ldots, t_d)$ and $\D_0(\ell)(t_1^{-1}, \ldots, t_d^{-1})$ are equivalent. 

We see that Kishino's knot $k$ is not invertible as follows. If $k$ were invertible, then there would exist an symplectic inversion of $\R^4 = {\rm span}(x, y, u, v)$ such that $x - uvx \mapsto -v+vxy$. 

If $x \mapsto -v$ then $uvx \mapsto -vxy = -y(-x)(-v)$ and hence $y \mapsto -u$. But 
the symplectic pairing $\<x, y\>$ is equal to $1$ while $\<-v, -u\> = -1$, a contradiction. (In fact, this change of basis corresponds to flipping $(S \times I, k)$ over, reversing the orientations of both the knot $k$ and the surface $S$.) 

The  only other  possibility is $x \mapsto vxy$ and $v \mapsto uvx$. In this case, 
$uvxy = (uvx)(vxy)x^{-1}v^{-1} \mapsto v x (vxy)^{-1}(uvx)^{-1} = u^{-1}v^{-1}x^{-1} y$. Since the middle term of $\D_0(k)$ is not preserved, we again have a contradiction. 

Kauffman informs the authors that the noninvertibility of Kishino's knot also can be shown using the 
parity bracket \cite{kauffcomm}. 

\end{example}


\begin{example} A. Stoimeow proposed the virtual link $\tilde \ell$ in Figure \ref{stoimenow} as an example for which the methods of \cite{dk05} appear to be insufficient to determine virtual genus. 

Instead of computing directly, we can recognize that $\tilde \ell$ is a satellite and use \cite{sw12}. In the companion link $\ell$, the classical trefoil component is replaced by an unknot. We simplify further by computing the one-variable polynomial  $\D_0(\ell)(t, t)$, which is equal to
$$(t-1)^2[(yx^{-1}-1)+ t(1-y)+t^2(-1+2y-y^2)+t^3(y^2-y)+t^4(xy-y^2)].$$
The symplectic rank is 2. The symplectic rank of $\D_0(\ell)(t_1, t_2)$ cannot be smaller.
Since the link has a diagram on a torus, Corollary \ref{reduce} implies that $\ell$ and hence $\tilde \ell$ have virtual genus equal to $1$. 

\end{example}


\end{document}